\documentclass{article}  
                                                       
\usepackage{amsmath} 
\usepackage{acronym}
\usepackage{xcolor}
\usepackage{amsmath}
\usepackage{amssymb}
\usepackage{mathtools}
\usepackage{subcaption}
\usepackage{hyperref}
\hypersetup{colorlinks=true, linkcolor=black}
\usepackage[ruled]{algorithm2e}
\usepackage{mathrsfs}
\newcommand{\rev}[1]{{\color{black} #1}}
\newcommand{\rw}[1]{{\color{black} #1}}

\newcommand{\R}{\mathbb{R}}

\newcommand{\0}{\mathbf{0}}
\newcommand{\1}{\mathbf{1}}
\newcommand{\psd}{\mathbb{S}}

\DeclareMathOperator*{\find}{find}

\DeclarePairedDelimiter{\norm}{\lVert}{\rVert}

\DeclarePairedDelimiter{\Tr}{\textrm{Tr}(}{)}

\DeclarePairedDelimiterX{\inp}[2]{\langle}{\rangle}{#1, #2}

\usepackage{amsthm}
\newtheorem{thm}{Theorem}[section]
\newtheorem{lem}[thm]{Lemma}
\newtheorem{prop}[thm]{Proposition}
\newtheorem{prob}[thm]{Problem}

\newtheorem{defn}{Definition}[section]

\newtheorem{rmk}{Remark} 
\usepackage[margin=1in]{geometry}
\newcommand{\kcol}{\otimes_{\textrm{col}}}

\newcommand{\bth}{\mathbf{\Theta}}
\newcommand{\bx}{\mathbf{X}}
\newcommand{\bu}{\mathbf{U}}
\newcommand{\bw}{\mathbf{W}}

\newcommand{\xp}{\mathbf{X}_+}
\newcommand{\xn}{\mathbf{X}_-}
\newcommand{\xd}{\mathbf{X}_\delta}

\newcommand{\dc}{\mathcal{D}}

\newcommand{\aell}{\{A_\ell\}}

\DeclarePairedDelimiter{\qbp}{\textrm{QBP}(}{)}
\DeclarePairedDelimiter{\sqbp}{\textrm{SQBP}(}{)}

\newcommand{\kron}{\otimes}

\newcommand{\allassum}{A1-A5}

\usepackage{harmony}

\usepackage{color}
\usepackage{ulem}

\title{\LARGE \bf Data-Driven Gain Scheduling Control  \\ of Linear Parameter-Varying Systems  \\ using Quadratic Matrix Inequalities
}

\author{Jared Miller$^1$,  Mario Sznaier$^1$
\thanks{$^1$J. Miller, and M. Sznaier are with the Robust Systems Lab,  ECE Department, Northeastern University, Boston, MA 02115. (e-mails: miller.jare@northeastern.edu, msznaier@coe.neu.edu).}
\thanks{J. Miller and M. Sznaier were partially supported by NSF grants  CNS--1646121, ECCS--1808381 and CNS--2038493, AFOSR grant FA9550-19-1-0005, and ONR grant N00014-21-1-2431.  
J. Miller was in part supported by the Chateaubriand Fellowship of the Office for Science \& Technology of the Embassy of France in the United States.}}

\begin{document}

\maketitle
\thispagestyle{empty}
\pagestyle{empty}


\begin{abstract}
\label{sec:abstract}
This paper synthesizes a gain-scheduled controller to stabilize all possible Linear Parameter-Varying (LPV) plants that are consistent with measured input/state data records. Inspired by prior work in data informativity and LTI stabilization, a set of Quadratic Matrix Inequalities is developed to represent the noise set, the class of consistent LPV plants, and the class of stabilizable plants. The bilinearity between unknown plants and `for all' parameters is avoided by vertex enumeration of the parameter set. Effectiveness and computational tractability of this method is demonstrated on example systems.






\end{abstract}
\section{Introduction}
\label{sec:introduction}



This paper performs \ac{DDC} of discrete-time \ac{LPV} systems using \acp{QMI}. The problem setting involves parameter-affine \ac{LPV} systems in which the parameter may vary arbitrarily within a polytope and the measured data admits a quadratic description in its noise. When the system has $n$ states, $m$ inputs, $L$ parameters, and $N_v$ vertices in the parameter polytope, we propose a non-conservative \ac{LMI} to find a quadratically stabilizing gain-scheduled controller for all consistent \ac{LPV} plants involving $N_v$ \ac{PSD} constraints of size $n(L+1) + m$ (continuous-time) or $n(L+2) + m$ (discrete-time) and a single \ac{PD} constraint of size $n$.



\acp{LPV} systems are a class of linear systems whose plant dynamics depend on externally measured parameters. \ac{LPV} systems have been employed to model and control nonlinear dynamics such as in vehicle control \cite{sename2013robust}, missile control \cite{apkarian1995self}, and chemical processes \cite{bachnas2014review}. Gain-scheduling control sets the input to be a function of the state and measured parameter \cite{rugh2000research}. 
%
Examples of quadratically stabilizing gain-scheduling through a common Lyapunov function include backsubstitution \cite{becker1993control}, interpolated vertex-controllers when the \ac{LPV} dynamics are parameter-affine \cite{apkarian1995self}, and the use of a dynamic compensator when the plant dynamics are a Linear Fractional Transformation of the applied parameter \cite{packard1994gain}. The work in \cite{scherer2001lpvqmi} applied different \acp{QMI} for robust control of a single given continuous-time \ac{LPV} plant.



\ac{DDC} is a methodology of formulating controllers for all possible plants that are consistent with measured input/output relations (data) \cite{HOU20133}. Such algorithms avoid an  expensive system-identification step to construct a generalized plant model. A survey of data-driven techniques is provided in \cite{hou2017datasurvey}. 
One class of \ac{DDC} methods applies Willem's Fundamental Lemma, which parameterizes all possible system responses by linear combinations \rw{of} a single trajectory's Hankel matrices if a rank condition is satisfied (persistency of excitation) \cite{willems2005note}. 
This Lemma can be used for stabilization/regulation \cite{depersis2020formulas} and  Model Predictive Control \cite{coulson2019data, berberich2020mpc} with optional regularization to reduce sensitivity to noise.




When the noise corrupting the recorded data admits a quadratic description, \acp{QMI} may be used in a non-conservative manner to describe the noise set and \rw{the} set of consistent plants \cite{waarde2020noisy}. Their work forms a matrix S-Lemma \cite{yakubovich1997s},
providing conditions under which the satisfaction of one \ac{QMI} implies another \ac{QMI} \cite{waarde2022qmi}, in order to perform quadratic stabilization and robust control ($H_2$ and $H_\infty$). The \ac{QMI}-with-S-Lemma approach has also been used to stabilize nonlinear systems with state-dependent representations \cite{dai2021nonlinear}, to form a robust-control framework incorporating prior knowledge \cite{berberich2020combining}, to analyze and control continuous-time systems \cite{BERBERICH2021210},  to iteratively stabilize networked systems with block-structured controllers  \cite{eising2022informativity}, \rev{and to impose \ac{LMI}-region performance constraints on robust controllers  \cite{bisoffi2022learning}}.


\ac{DDC} has been previously applied to \ac{LPV} systems, as surveyed by \cite{bachnas2014review}. Other instances of \ac{DDC} for \ac{LPV} include using Support Vector Machines \cite{formentin2016direct}, hierarchical control \cite{piga2018direct}, and Willem's Fundamental Lemma \cite{verhoek2021fundamental}. \rev{The related problem of DDC of switched systems was studied in \cite{dai2018moments} using polynomial optimization}. To the best of our knowledge, \acp{QMI} and the matrix S-Lemma have not been used for \rev{\ac{DDC} of \ac{LPV} systems}.


    

The contributions of our work are
\begin{itemize}
    \item A presentation of the Data-Driven \ac{LPV} quadratic stabilization problem parameterized by \acp{QMI}
    \item An \ac{LMI} to achieve quadratic stabilization via gain-scheduling vertex-\acp{QMI} with Kronecker structure in continuous-time and discrete-time
    \item An accounting of computational complexity which includes allowances for sparsity
\end{itemize}

This paper has the following structure: 
Section \ref{sec:preliminaries} reviews preliminaries such as notation, \ac{LPV} stabilization, and the use of \acp{QMI} in forming stabilizing controllers. Section \ref{sec:lpv_qmi} applies this \ac{QMI} method for \ac{LPV} stabilization. 
Section \ref{sec:h2_control} performs worst-case suboptimal $H_2$ control on LPV plants consistent with the noise structure.
Section \ref{sec:examples} demonstrates this stabilization approach on example systems. Section \ref{sec:conclusion} concludes the paper.
\section{Preliminaries}
\label{sec:preliminaries}

\begin{acronym}[LPVA]
\acro{DDC}{Data Driven Control}

\acro{LMI}{Linear Matrix Inequality}
\acroplural{LMI}[LMIs]{Linear Matrix Inequalities}
\acroindefinite{LMI}{an}{a}


\acro{LPV}{Linear Parameter-Varying}
\acroindefinite{LPV}{an}{a}

\acro{LPVA}{LPV A-affine}
\acroindefinite{LPVA}{an}{a}




\acro{PD}{Positive Definite}

\acro{PSD}{Positive Semidefinite}

\acro{SDP}{Semidefinite Program}
\acroindefinite{SDP}{an}{a}


\acro{QMI}{Quadratic Matrix Inequality}
\acroplural{QMI}[QMIs]{Quadratic Matrix Inequalities}

\end{acronym}

\subsection{Notation}
The double dots in $1..L$ represent the sequence of natural numbers between $1$ and $L$.
The $n$-dimensional real vector space is $\R^n$. The nonnegative real orthant is $\R^n_{\geq 0}$ and the cone of positive vectors is $\R^n_{>0}$. The set of real-valued $m \times n$ matrices is $\R^{m \times n}$. The transpose of a matrix $M$ is $M^T$. The kernel (nullspace) of a matrix $M$ is $\textrm{ker}(M)$.
The set of symmetric matrices of size $n$ is $\psd^n$, and
its subsets of \ac{PSD} and \ac{PD} matrices are $\psd_+^n$ and $\psd_{++}^n$. The vertical concatenation of matrices $A$ and $B$ of compatible dimensions is $[A; B]$ and their horizontal concatenation is $[A, B]$. The symmetrization operator applied to $M \in \R^{n \times n}$ is $\textbf{sym}(M) = (M+M^T)/2$. \rw{The pseudoinverse of a matrix $M$ is $M^\dagger$.}

The matrices $I_{n}, \ \0_{m \times n}, \ \1_{m \times n}$ are respectively the identity, zeros, and ones matrices of appropriate dimensions. 
The dimension subscripts will be dropped when the matrix sizes are unambiguous. The $\ast$ marking will be used in block matrices to refer to the canonical transpose of oppositely-indexed elements. The Kronecker product of matrices $P$ and $Q$ is $P \otimes Q$. The Hadamard (elementwise) product of matrices is $P \odot Q$. The symbol $\kcol$ will denote the column-wise Khatri-Rao product for matrices $A \in \R^{m \times n}, \ B \in \R^{p \times n}$
\cite{khatri1968solutions}
\begin{equation}
    \label{eq:kron_col}
    A \kcol B = (\1_{p \times 1} \otimes A) \odot (B \otimes \1_{m \times 1}).
\end{equation}

\rev{The convex hull of a set of points $P = \{p_j\}^{N}_{j=1}$ is $\textrm{conv}(P)$.} The notation $\delta x$ will mean the derivative $\dot{x}$ in continuous-time or the next state $x_+$ in discrete-time.


\subsection{LPV Stabilization}

\ac{LPV}  dynamics with state $x \in \R^n$, input $u \in \R^m$, and measurable parameter $\theta \in \Theta \subset \R^L$ are
\begin{align}    \delta x &= A(\theta) x + B(\theta) u. \label{eq:lpv_general}\\
\intertext{The \acf{LPVA} structure \cite{besselmannlofberg2012}  has $B$ constant and $A$ $\theta$-affine
for some set of matrices $\forall \ell: A_\ell \in \R^{n \times n}$ if}
    \delta x &= \textstyle \left( \sum_{\ell=1}^L A_\ell \theta_\ell\right )x + B u. \label{eq:LPVA}
\end{align}

\rev{This preliminary subsection will deliver exposition on the case where $(\{A_\ell\}, B)$ are known and fixed while $\theta$ is unknown and measured on-line. The main body of the paper will focus on the setting where the plant $(\{A_\ell\}, B)$ is unknown but consistent with observed data.}

\begin{rmk}
\ac{LPVA} structure may be rendered affine in the parameter by adjoining a new constant  $\theta_0 = 1$ to $\theta$.  
\end{rmk}


\rev{Let $\Omega = \{\omega_v\}_{v=1}^{N_v}$ be a finite set of $N_v$ points in $\R^L$. In this paper, the parameter set $\Theta$ will be chosen to be the compact convex polytope $\Theta = \textrm{conv}(\Omega)$. We will refer to $\Omega$ as the vertices of $\Theta$ (or as vertices more generally).}


\rev{
A vertex-controller $K_v \in \R^{m \times n}$ is defined at each vertex $\omega_v$ in $\Omega$, yielding the state-feedback law $u = K_v x$. Given a parameter $\theta \in \R^L$, a gain-scheduled controller $u=K(\theta) x$ may be found by first solving for a feasible $c \in \R^{N_v}$ using Linear Programming
\begin{subequations}
\label{eq:gain_sched}
\begin{align}
    \textrm{find} \ c & \in \R^{N_w}_+ &  \textstyle \sum_{v=1}^{N_v} c_v & = 1 &\textstyle   \sum_{v=1}^{N_v} c_v \omega_v & = \theta,
    \label{eq:interp_vert}
\end{align}
and then returning the control policy,
\begin{align}
     K\rev{(\theta)} &= \textstyle \sum_{v=1}^{N_v} c_v K_v & u &= K(\theta) x. \label{eq:gain_sched_out}
\end{align}
\end{subequations}

Any feasible point $c$ of \eqref{eq:interp_vert} will serve: uniqueness of $K(\theta)$ is not required.
}
Application of the gain-scheduled $u = K(\theta) x$ to \rev{the} \ac{LPVA} system \eqref{eq:LPVA} leads to \rev{the} decomposed dynamics
\begin{subequations}
\label{eq:dyn_controlled}
\begin{align}
    \delta x &= A(\theta) x + B K(\theta) x \\
    &= \textstyle \sum_{\ell=1}^L \theta_\ell A_\ell x  + \sum_{v=1}^{N_v} c_v B K_v x \\
    &= \textstyle \left[\sum_{v=1}^{N_v} c_v \left(\sum_{\ell=1}^L \omega_{\ell v} A_\ell\right) + c_v B K_v \right] x.
\end{align}
\end{subequations}
The open-loop system $A_v$ for each vertex $\omega_v$ (multiplied in \eqref{eq:dyn_controlled} by $c_v$) may be defined as
\begin{equation}
\label{eq:av}
    A_v = \textstyle  \sum_{\ell=1}^L \omega_{\ell v} A_\ell.
\end{equation}

\begin{lem}
\label{lem:cone_vert}
If $C$ is a convex cone with elements $z$ and $\Theta \rev{= \textrm{conv}(\Omega)}$ , then the following statements are equivalent:
\begin{subequations}
\begin{align}
    \textstyle \sum_{\ell=1}^{L} \theta_\ell z_\ell & \in C & & \forall \theta \in \Theta \label{eq:cone_vert_th}\\
    \textstyle \sum_{\ell=1}^{L} \omega_{\ell v} z_\ell & \in C & & \forall v = 1..N_v \label{eq:cone_vert_v}
\end{align}
\end{subequations}
\end{lem}
\begin{proof}
Statement \eqref{eq:cone_vert_th} implies \eqref{eq:cone_vert_v} because each vertex $\omega_v$ is an element of $\Theta$.
\rev{Every} point $\theta \in \Theta$ may be represented by a possibly non-unique convex combination of vertices with coordinates $\theta_\ell = \sum_{v=1}^{N_v} c_v \omega_{\ell v}$ \rev{given} that $\Theta = \textrm{conv}( \Omega)$  (\eqref{eq:interp_vert} and Section 2.1.4 of \cite{boyd2004convex}). \rev{Eq.} \eqref{eq:cone_vert_v} implies \eqref{eq:cone_vert_th}, because $\sum_{\ell=1}^{L} \theta_\ell z_\ell$ may be expressed as the convex combination of $C$-elements $\sum_{\ell=1}^{L} \sum_{v=1}^{N_v} \left(c_v \omega_{\ell v} \right) z_\ell$.
\end{proof}
\rw{
\begin{defn}
The controller $\rev{u=}K(\theta)\rev{x}$ from  \rev{Eq. \eqref{eq:gain_sched}} \textit{quadratically stabilizes} the \ac{LPVA} system \eqref{eq:LPVA} 
if there exists a $\theta$-independent $Y \in \psd_{++}^n$ (for continuous-time) or a $P \in \psd^n_++$ (for discrete-time)
\begin{subequations}
\begin{align}
    -2 \ \textbf{sym}(Y (A(\theta) + B  K(\theta))) &\in \psd_{++}^{n} & & \forall \theta \in \Theta \label{eq:lpv_stab_noiseless_th_cont}\\
    \begin{bmatrix} P & (A(\theta) + B K(\theta))P \\ \ast &  P \end{bmatrix} &\in \psd_{++}^{2n} & & \forall \theta \in \Theta \label{eq:lpv_stab_noiseless_th}
    \end{align}
    \end{subequations}
\end{defn}

\begin{lem}
\label{lem:vert_stab}
Equations \eqref{eq:lpv_stab_noiseless_th_cont} and \eqref{eq:lpv_stab_noiseless_th} are equivalent to the following respective conditions,
\begin{subequations}
\label{eq:lpv_stab_noiseless_vert}
\begin{align}
    -2 \ \textbf{sym}(Y(A_v + B  K_v)) &\in \psd_{++}^{n}   & & \forall v=1..N_v \label{eq:lpv_stab_noiseless_kv_cont}\\
    \begin{bmatrix} P & (A_v + B K_v) P \\ \ast &  P \end{bmatrix}  &\in \psd_{++}^{2n}   & & \forall v=1..N_v \label{eq:lpv_stab_noiseless_kv}
\end{align}
  \end{subequations}
 
  \end{lem}}
    \begin{proof}
  Equivalence of the \rev{respective pairs [\eqref{eq:lpv_stab_noiseless_th_cont}, \eqref{eq:lpv_stab_noiseless_kv_cont}] and [\eqref{eq:lpv_stab_noiseless_th}, \eqref{eq:lpv_stab_noiseless_kv}] } holds by Lemma \ref{lem:cone_vert} with \rev{regard} to the cones $\psd^{n}_{++}$ and $\psd^{2n}_{++}$ \rev{\cite{apkarian1995self}}. 
  \end{proof}
  Pre- and post-multiplying  \eqref{eq:lpv_stab_noiseless_kv_cont} by $Y^{-1}$ yields
  \begin{align}
          -2 \ \textbf{sym}((A_v + B  K_v)Y^{-1}) &\in \psd_{++}^{n}   & & \forall v=1..N_v. \label{eq:lpv_stab_noiseless_kv_cont_inv}
  \end{align}
  Problems \eqref{eq:lpv_stab_noiseless_kv_cont} and \eqref{eq:lpv_stab_noiseless_kv} are convex after substituting $S_v = K_v Y^{-1}$ (using \eqref{eq:lpv_stab_noiseless_kv_cont_inv}) and $S_v = K_v P$ respectively \cite{boyd1994linear}.



\subsection{Quadratic Matrix Inequalities}

This section reviews \acp{QMI} and the matrix S-Lemma approach proposed by \cite{waarde2020noisy, waarde2022qmi}.
\begin{defn}
\label{def:qmi}
Given a matrix $M \in \psd^{n}$, a \textit{QMI} is the \rev{quadratic statement in $X \in \rev{\R}^{n \times k}$}  that $X^T M X \in \psd_{+}^{k}$.
\end{defn}
\acp{QMI} can also be strict with $X^T M X \in \psd_{++}^k$.
    The works in \cite{waarde2020noisy, waarde2022qmi} present conditions under which one \ac{QMI} implies another \ac{QMI}, with specific attention on the scenario where $X$ can be partitioned as $X = [I; Z^T]$ for some $Z$. 
\rev{In this case, the variable $Z$ is referred to as satisfying a QMI constraint.} 

\begin{defn}

\rev{Let $\Phi \in \psd^{n + k}$ be a partitioned matrix,}
\begin{subequations}
\label{eq:qbp}
\begin{align}
    & \Phi_{11} \in \psd^n, \  -\Phi_{22} \in \psd^k_{\rw{+}}. \label{eq:phi_condition}
     \intertext{A matrix $Z \in \R^{n \times k}$ satisfies the \textit{Quadratic Boundedness Property} with respect to \rev{$\Phi$}  ($Z \in \qbp{\Phi}$)  if}
& \begin{bmatrix} I_n \\ Z^T\end{bmatrix} ^T \begin{bmatrix} \Phi_{11} & \Phi_{12} \\ \Phi_{12}^T & \Phi_{22} \end{bmatrix} 
    \begin{bmatrix} I_n \\ Z^T\end{bmatrix} \in \psd_{+}^{k}. \label{eq:qbp_psd}
\end{align}
\end{subequations}
\end{defn}
\rev{
\begin{lem}[Theorem 3.2b of \cite{waarde2022qmi}]
Assuming that $\Phi$ satisfies \eqref{eq:phi_condition}, let $\Phi \mid \Phi_{22}$ be the \rw{Generalized} Schur complement $\Phi_{11} - \Phi_{12} \Phi_{22}^{\dagger} \Phi_{12}$,  $\norm{\cdot}_F$ be the Frobenius norm, and $\lambda_{\max}$ ($\lambda_{\min}$) be the maximum (minimum) matrix eigenvalue. Then for all matrices $ Z \in \qbp{\Phi}:$
\begin{equation*}
\quad \norm{Z + \Phi_{22}^{-1} \Phi_{12}}_F^2 < k \lambda_{\max}(\Phi \mid \Phi_{22}) / \lambda_{\min}(-\Phi_{22}).
\end{equation*}
\rw{$Z$ is therefore bounded if $-\Phi_{22} \in \psd^k_{++}$}.
\end{lem}
}

\begin{defn}
The \textit{Strict Quadratic Boundedness Property} ($Z \in \sqbp{\Phi}$) \rev{holds} if the matrix in \eqref{eq:qbp_psd} is \rev{in $\psd_{++}^k$}.
 \end{defn}
Structures of $\Phi$ are listed in Section 2 of \cite{waarde2022qmi}. Particular instances include energy bounds $\Phi_{11} - Z Z^T \in \psd_{+}^n $ (with $\Phi_{12} =\0, \ \Phi_{22} = -I_k$) and individual sample $L_2$ bounds (adding  some conservatism) $\forall k'=1..k: \ \norm{z_{k'}}_2 \leq \epsilon,$ (with $\Phi_{11} = \epsilon^2 k I_n, \Phi_{12} = \0, \Phi_{22} = -I_k$).

\begin{thm}[Strict Matrix S-Lemma, \rev{[Cor. 4.13 of \cite{waarde2022qmi}]}]
\label{thm:slemma}
Let $M, N \in \psd^{m+\rev{k}}$ be matrices \rev{satisfying \eqref{eq:phi_condition}} with the same partitioning scheme and let $Z \in \R^{n \times k}$. The following conditions are equivalent under the assumptions that 
$\textrm{ker}N_{22} \subseteq \textrm{ker}N_{12}$,  \rw{$N \mid N_{22} \in \psd_+^n$, and $-M_{22} \in \psd^k_{++}$}:
\begin{subequations}
    \begin{align}
        &Z \in \sqbp{M}, \quad    \forall Z \in \qbp{N} \\
        &\exists \alpha \geq 0, \beta> 0: \\
        &\quad M - \alpha N - \begin{bmatrix}
        \beta I_m & \0 \\ \0 & \0_{k \times k}
        \end{bmatrix} \in \psd_{\rw{+}}^{m+k}. \nonumber
    \end{align}
\end{subequations}
\end{thm}


\section{LPV Stabilization with QMIs}
\label{sec:lpv_qmi}

\subsection{Problem Description}

\rev{A sampling process records a set of $T$ observations from an unknown \ac{LPVA} \rw{system} \eqref{eq:LPVA} under a bounded noise process $w(\cdot)$ (discrepancy) for $t=0..T$
\begin{equation}
\delta x(t) = \textstyle \left( \sum_{\ell=1}^L A_\ell \theta_\ell\right )x(t) + B u(t) 
+ w(t). \label{eq:LPVA_W}
    \end{equation}
This data is collected into matrices $(\bx_-, \bu, \bth)$ }
\begin{align}
\label{eq:data}
    \begin{array}{cccccr}
        \xn & := & [x(0) & x(1) &  \ldots & x(T-1) ]   \\
        \bu & := & [u(0) & u(1) & \ldots & u(T-1)] \\ 
        \bth & := & [\theta(0) & \theta(1) & \ldots & \theta(T-1)].
    \end{array}
\end{align}

The derivative observations $\dot{\bx}$ (continuous-time) and one-step-ahead records $\xp$ (discrete-time) are
\begin{align}
\begin{array}{cccccr}
    \dot{\bx} & := & [\dot{x}(0) & \dot{x}(1) &  \ldots & \dot{x}(T-1)] \\
    \xp & := & [x(1) & x(2) &  \ldots & x(T)].
    \end{array}
\end{align}

The symbol $\xd$ will refer to $\dot{\bx}$ or $\xp$ as appropriate. The data $\dc$ will denote the tuple $(\bx_\rev{-}, \bu, \bth, \rev{\xd})$.

Let $\bth_\ell \in \R^{1 \times T}$ be the row of $\bth$ associated with parameter $\theta_\ell$.
The discrepancy $\bw$ \rev{collected from \eqref{eq:LPVA_W}} (mathematically equivalent to process noise for discrete-time) associated with the observations in $\dc$ for a given LPVA $(A(\theta), B)$ is
\begin{align}
\label{eq:lpv_noise}
    \bw &= \xd - \textstyle \left(\sum_{\ell=1}^L \bth_\ell \kcol A_\ell\right) \xn- B \bu.
\end{align} 

The following assumptions will be imposed,
\begin{itemize}
    \item[A1] $n, m, L, T$ are all finite and known.
    \item[A2] The set $\Theta$ is a known compact non-empty polytope with vertices $\Omega$. 
    \item[A3] The ground truth system has LPVA structure \eqref{eq:LPVA}.
    \item[A4] There exists a known $\Phi \in \psd^{n + T}$ \rev{satisfying \eqref{eq:phi_condition}}
    such that $\bw \in \qbp{\Phi}$ for the ground-truth system.
\end{itemize}

The consistency set of plants $(A(\theta), B)$ compatible with $\dc$ given $\Phi$ is
\begin{align*}
    \Sigma_{\dc}(\Phi) = \{(\{A_\ell\}_{\ell=1}^L, B) \mid \bw \ \rev{\textrm{from  \eqref{eq:lpv_noise}}}\in \qbp{\Phi}\}.
\end{align*}

\begin{rmk}
Data matrices arising from multiple trajectories may be horizontally concatenated if the noise structure in $\Phi$ is compatible with the arrangement (Example 2 of \cite{waarde2020informativity}).
\end{rmk}

Our goal is to solve the following problem,
\begin{prob}
\label{prob:lpv_qmi}
Find a gain-scheduled \rev{(Eq. \eqref{eq:gain_sched})} control policy $u = K(\theta) x$ such that $x_+ = (A(\theta) +  B K(\theta)) x$ is quadratically stable for all $(\{A_\ell\}, B) \in \Sigma_\dc$.
\end{prob}

\begin{rmk}
\label{rmk:solve_cone}
Problem \eqref{prob:lpv_qmi} will be solved by enforcing that \eqref{eq:lpv_stab_noiseless_vert} holds for all $(\{A_\ell\}, B) \in \Sigma_\dc$
\rw{(Lemma \ref{lem:vert_stab}).}
\end{rmk}

\subsection{Data Consistency QMI}
The set $\Sigma_\dc(\Phi)$ may be represented as \iac{QMI}.

Using the convention that $\{A_\ell\} = [A_1, A_2, \ldots, A_L]$  and  $\{A_\ell^T\} =  [A_1^T ; A_2^T; \ldots; A_L^T]$, the \rev{discrepancy} matrix $\bw$ from \eqref{eq:lpv_noise} may be represented as
\begin{equation}
    \begin{bmatrix} I_n \\ \bw^T\end{bmatrix} = \begin{bmatrix}
    I_n & \xd \\
    \0_{n \times Ln} & -\bth \kcol \xn \\
    \0_{n \times m} & -\bu
    \end{bmatrix}^T
    \begin{bmatrix}
    I_n \\ \{A_\ell^T\} \\ B^T
    \end{bmatrix}.
\end{equation}
Defining the matrix $\Psi \in \psd^{n + (Ln + m)}$ as
\begin{equation}
\label{eq:psi_matrix}
    \Psi = \begin{bmatrix}
    I_n & \xd \\
    \0_{n \times Ln} & -\bth \kcol \xn \\
    \0_{n \times m} & -\bu
    \end{bmatrix} \Phi \begin{bmatrix}
    I_n & \xd \\
    \0_{n \times Ln} & -\bth \kcol \xn \\
    \0_{n \times m} & -\bu
    \end{bmatrix}^T,
\end{equation}
it holds that the following two descriptions are identical:
\begin{align}
    (\aell, B) &\in \Sigma_\dc(\Phi) & \leftrightarrow & &  [\aell, B] &\in \qbp{\Psi}. \label{eq:qmi_data}
 \end{align}

\subsection{Stabilization QMI}
This section will form \iac{QMI} for stabilization of the subsystem $A_v \in \R^{n \times n}$ at vertex $v$ from \eqref{eq:av} by a controller $K_v \in \R^{m \times n}$.  The continuous-time \ac{LMI} criterion in  \eqref{eq:lpv_stab_noiseless_kv_cont_inv} is equivalent to the following \ac{QMI}
\begin{align}
\label{eq:qmi_stab_k_cont}
[\aell, B] \in \textrm{SQBP} \begin{pmatrix}
    \0 & * & * \\
 -\omega_v \kcol Y^{-1} & \0 & *\\
    -K_v Y^{-1} & \0  & \0
    \end{pmatrix},
\end{align}
as obtained by pre- and post-multiplying \eqref{eq:lpv_stab_noiseless_kv_cont} by the invertible $Y^{-1} \in \psd^n_{++}$.
The discrete-time \ac{LMI} criterion in \eqref{eq:lpv_stab_noiseless_kv} is equivalent to the following \ac{QMI} by collecting terms
\begin{equation}
\label{eq:qmi_stab_k}
    [\aell, B] \in \textrm{SQBP} \begin{pmatrix}
    P & * & * \\
    \0 & -(\omega_v \omega_v^T) \kron P & *\\
    \0 & -(\omega_v^T) \kron (K_vP) & -K_v P K_v^T
    \end{pmatrix}.
\end{equation}



\begin{thm}[Continuous-Time]
\label{thm:stab_vert_cont}
\rev{Under assumptions \allassum,} \ac{QMI} \eqref{eq:qmi_stab_k_cont} holds for all $(\aell, B) \in \Sigma_\dc(\Phi)$  if and only if $\exists \alpha_v \geq 0, \beta_v > 0$ \rev{such that}
\begin{align}
\begin{bmatrix}
        -\beta_v I_n & \ast & \ast \\
    -\omega_v \kcol Y^{-1} &\0 & \ast \\
    -K_v Y^{-1} & \0 & \0
\end{bmatrix} - \alpha_v \Psi & \in \psd_{+}^{(L+1)n + m} \label{eq:qmi_k_s_cont}.
\end{align}
\end{thm}
\begin{proof}
This will follow a similar proof strategy as Sections IV of \cite{waarde2020noisy} \rev{ and V.I of \cite{waarde2022qmi}.}
The $(\alpha_v, \beta_v)$ structure follows from Theorem \ref{thm:slemma}. It remains to affirm the assumptions under which this theorem is valid. 
Given that $Y \in \psd_{++}^n$ and $-\Phi_{22} \in \psd_{+}^T$ (A4), the lower-right corner of the matrix in \eqref{eq:qmi_stab_k_cont} and $\Psi$ may each be expressed as
\begin{subequations}
\begin{align}
    -\0_{(L+1)n + m} &\in \psd_{+}^{Ln + m} \\
    -\begin{bmatrix}
    \bth \kcol \xn \\ \bu
    \end{bmatrix} \Phi_{22} \begin{bmatrix}
    \bth \kcol \xn \\ \bu
    \end{bmatrix}^T &\in \psd_{+}^{Ln + m}. 
\end{align}
\end{subequations}
The final condition is that $\textrm{ker} \Psi_{22} \subseteq \textrm{ker} \Psi_{12}$ with
\begin{subequations}
\begin{align}
    \textrm{ker} \Psi_{22} &=  \textrm{ker} \begin{bmatrix}
    \bth \kcol \xn \\ \bu
    \end{bmatrix} \\
    \textrm{ker} \Psi_{12} &= \textrm{ker} (\Phi_{12} + \Phi_{22} \xd) \begin{bmatrix}
    \bth \kcol \xn \\ \bu
    \end{bmatrix}.
\end{align}
\end{subequations}
All conditions are satisfied, so Theorem \rw{\ref{thm:stab_vert_cont}} is proven.
\end{proof}

\begin{thm}[Discrete-Time]
\label{thm:stab_vert}
\rev{Under assumptions \allassum,} \ac{QMI} \eqref{eq:qmi_stab_k} is satisfied  $ \forall (\aell, B) \in \Sigma_\dc(\Phi)$  if and only if $\exists \alpha_v \geq 0  \beta_v > 0$ \rev{such that}
\begin{align}
& \begin{bmatrix}
\label{eq:qmi_k_s}
        P-\beta_v I_n & \ast & \ast \\
    \0 & -(\omega_v \omega_v^T) \kron P & \ast \\
    \0 & -(\omega_v^T) \kron (K_v P) & -K_v P K_v^T
\end{bmatrix} - \alpha_v \Psi \nonumber\\
& \qquad \qquad \in \psd_{+}^{(L+2)n + m}.
\end{align}
\end{thm}
\begin{proof}
This proof follows the same pattern as in the above Theorem \ref{thm:stab_vert_cont}. The only modification required is demonstrating that the negative of the lower right-corner matrix in \eqref{eq:qmi_stab_k} is \ac{PSD}, which holds by
\begin{equation}
    \begin{bmatrix}
    (\omega_v \omega_v^T) \kron P & \ast \\
    (\omega_v^T) \kron (K_v P) & K_v P K_v^T
    \end{bmatrix} =  \begin{bmatrix}
    \omega_v  \kron I_n \\
    K_v P 
    \end{bmatrix} P \begin{bmatrix}
    \omega_v  \kron I_n \\
    K_v P 
    \end{bmatrix}^T.
\end{equation}
All other conditions are valid, completing the proof.
\end{proof}

\subsection{Controller Generation Program}
\rev{This subsection will pose a pair of \acp{SDP} to solve
data-driven \ac{LPV} stabilization under continuous-time  and discrete-time, as introduced by Remark \ref{rmk:solve_cone} under assumptions \rev{\allassum}. In the language of \cite{waarde2020informativity}, the tuple $(\dc, \Phi, \Omega)$ is \textit{informative} for \ac{LPV} quadratic stabilization if the respective \ac{LMI}  is feasible.

%
}

\subsubsection{Continuous-Time}
The first matrix of \eqref{eq:qmi_k_s_cont} admits the substitution $P = Y^{-1}, \ S_v = K_v P$ to form the \ac{LMI}
\begin{equation}
    \begin{bmatrix}
        -\beta_v I_n & \ast & \ast \\
    -\omega_v \kcol P &\0 & \ast \\
    -S_v & \0 & \0
\end{bmatrix} - \alpha_v \Psi \in \psd_{+}^{(L+1)n + m} \label{eq:lmi_k_s_cont_y}.
\end{equation}

\rev{
The continuous-time stabilization \ac{SDP} with gain-scheduled control matrices $\{K_v = S_v P^{-1}\}_{v=1}^{N_v}$ is
\begin{subequations}
\label{eq:lpv_qmi_alg_c}
\begin{align}
  \textrm{find} \   & P \in \psd_{++}^n \label{eq:lpv_qmi_alg_p}\\
    & \alpha \in \R_{\geq 0}^{N_v}, \ \beta \in \R_{> 0}^{N_v}, \\
    & S_v \in \R^{m \times n} & & \forall v = 1..N_v \\
    & \textrm{LMI  \eqref{eq:lmi_k_s_cont_y} holds} & & \forall v = 1..N_v. \label{eq:lpv_qmi_alg_gamma}
\end{align}
\end{subequations}
}
\subsubsection{Discrete-Time}
The first matrix in \eqref{eq:qmi_k_s} may be expressed using a substitution $S_v = K_v P$ (with $K_v P K_v^T =S_v P^{-1} S_v^T$)
\begin{align}
& \begin{bmatrix}
    P-\beta_v I & \0 & \0 \\
    \0 & -(\omega_v \omega_v^T) \kron P & -\omega_v \kron (S_v^T) \\
    \0 & -(\omega_v^T) \kron S_v & S_v P^{-1} S_v^T
    \end{bmatrix}, \\
    \intertext{followed by a Schur Complement}
    \rightarrow & \begin{bmatrix}
    P-\beta_v I & \0 & \0  & \0 \\
    \0 & -(\omega_v \omega_v^T) \kron P & -\omega_v \kron (S_v^T)  & \0 \\
    \0 & -(\omega_v^T) \kron S_v & \0 & S_v \\
    \0 & \0 & S_v & P
    \end{bmatrix}. \label{eq:gamma_bv}
\end{align}
Letting $\Gamma_v(\beta_v)$ be the matrix in \eqref{eq:gamma_bv}\rw{,}
the \ac{LMI} \eqref{eq:qmi_k_s} from Theorem \ref{thm:stab_vert} may be restated as,
\begin{equation}
    \Gamma_v(\beta_v) - \alpha_v \begin{bmatrix}\Psi & \0 \\ \0 & \0_{n\times n}\end{bmatrix} \in \psd_{+}^{n(L+2)+m}. 
    \label{eq:gamma_bv_lmi}
\end{equation}



\rev{
The discrete-time stabilization \ac{SDP} with gain-scheduled control matrices $\{K_v = S_v P^{-1}\}_{v=1}^{N_v}$ is
\begin{subequations}
\label{eq:lpv_qmi_alg_d}
\begin{align}
  \textrm{find} \   & P \in \psd_{++}^n \label{eq:lpv_qmi_alg_p_d}\\
    & \alpha \in \R_{\geq 0}^{N_v}, \ \beta \in \R_{> 0}^{N_v}, \\
    & S_v \in \R^{m \times n} & & \forall v = 1..N_v \\
    & \textrm{LMI  \eqref{eq:gamma_bv_lmi} holds} & & \forall v = 1..N_v. \label{eq:lpv_qmi_alg_gamma_d}
\end{align}
\end{subequations}
}
\begin{rmk}
 In the specific discrete-time case  where $L = 1$ and $\Theta = \{\theta = 1\}$, \rev{Eq. \eqref{eq:lpv_qmi_alg_d}} is identical to \rev{Thm.} 14 of \cite{waarde2020noisy}.
\end{rmk}
\rev{
\begin{rmk}
Programs \eqref{eq:lpv_qmi_alg_c} and \eqref{eq:lpv_qmi_alg_d} can be normalized by constraining $\Tr{P}=1$.
\end{rmk}
}

\subsection{Computational Considerations}

The per-iteration complexity of solving \iac{SDP} using an interior point method up to arbitrary (nonzero) accuracy with a single \ac{PSD} variable of size $N$ with $M$ affine constraints is $O(N^3M + M^2 N^2)$ \cite{alizadeh1995interior}. 
\rev{The continuous-time \ac{SDP} in \eqref{eq:lpv_qmi_alg_c} has 1 \ac{PSD} constraint of size $n$ \eqref{eq:lpv_qmi_alg_p} and $N_v$ \ac{PSD} constraints of size $n(L+1)+m$ \eqref{eq:lpv_qmi_alg_gamma}. The discrete-time \ac{SDP} in  \eqref{eq:lpv_qmi_alg_d} has  1 \ac{PSD} constraint of size $n$ \eqref{eq:lpv_qmi_alg_p_d} and $N_v$ \ac{PSD} constraints of size  $n(L+2)+m$ \eqref{eq:lpv_qmi_alg_gamma_d}.}

The performance of \acp{SDP}  \eqref{eq:lpv_qmi_alg_c} and \eqref{eq:lpv_qmi_alg_d} therefore scales linearly in $N_v$, polynomially in $(n, L, m)$, and independently of $T$. 
Linear dependence on $N_v$ may result in an exponential scaling on $L$ (e.g. a hypercube  with $N_v = 2^L$).

\section{H2 Optimal Control}
\label{sec:h2_control}

A \rev{continuous}-time LPVA state-space system with external input $\rev{\xi} \in \R^{e}$  and regulated output $z \in \R^r$ given matrices $C \in \R^{r \times n}, \ D \in \R^{r \times m}, F \in \R^{n \times e}$ is
\begin{align}
    \dot{x} &= \textstyle \sum_{\ell=1}^L \theta_\ell A_\ell x + B u + F \rev{\xi}, &     z &= C x + Du. \label{eq:h2_sys}
\end{align}

\rev{The recorded data in $\dc$ has $\rev{\xi}=0$ while $\mathbf{W} \in \qbp{\Phi}$. The input $\rev{\xi}$ is applied during system execution.}

Define the $\rev{H}_2$ norm of  \eqref{eq:h2_sys} as \rev{
the worst-case (over all parameter trajectories) expected root-mean-square value of $\norm{z}_2$  when the input $\xi$ is a white noise process with identity covariance.
}
Then we have the following bound:
\begin{prop}
\label{prob:h2_single}\rev{There exists a gain-scheduled controller $u = K(\theta) x$ such that the closed-loop}  $H_2$ norm of \rev{the LPVA system} \eqref{eq:h2_sys} is bounded above by $\gamma \in \R_+$ if \rev{for all $v=1..N_v$} the following \ac{LMI} is feasible \cite{de2003parametric}
\begin{subequations}
\label{eq:h2_single}
\begin{align}
    \find_{P, Z, S} \quad &  \rev{-2 \ \textbf{sym}(A_v P + B S_v) - F F^T}
    \label{eq:h2_single_hurwitz}\in \psd_{++}^{\rev{n}} \\
    & \begin{bmatrix} Z & C P + D S_v \\ \ast & P \end{bmatrix} \in \psd_{++}^{n+r}  \label{eq:h2_z}\\
    & \Tr Z \leq \gamma^2 \\
    & P \in \psd^n_{++}, \ Z \in \psd^r_+, \ S_v \in \R^{m \times n} \label{eq:h2_def}.
\end{align}
\end{subequations}
\end{prop} 
The gain-scheduled controller $K(\theta)$ may be recovered from $\{\rev{\forall v: \ }K_v = S_v P^{-1}\}$ and \rev{Eq. \eqref{eq:gain_sched}}.
The variables $(Z, P)$ and given entries $(C, D, F)$ are independent of $(A, B) \in \Sigma_\dc$.
\begin{equation}
    \rev{\begin{bmatrix}
        -\beta_v I_n-F F^T & \ast & \ast \\
    -\omega_v \kcol P &\0 & \ast \\
    -S_v & \0 & \0
\end{bmatrix} - \alpha_v \Psi \in \psd_{+}^{(L+1)n + m} \label{eq:h2_lmi_gamma}.}
\end{equation}

Constraint \rev{ \eqref{eq:h2_lmi_gamma} is equal to \eqref{eq:lmi_k_s_cont_y}} when $F = \0_{n\times e}$, given that conditions \rev{\eqref{eq:h2_single_hurwitz} and \eqref{eq:lpv_stab_noiseless_kv_cont}} are identical under this restriction.

Worst-case $H_2$ control of \eqref{eq:h2_sys} for all $(A(\theta), B) \in \Sigma_\dc$ given $(C, D, F)$ may be conducted by solving
\begin{subequations}
\label{eq:lpv_qmi_h2}
\begin{align}
  \rev{\gamma^2 =} \inf & \quad  \Tr{Z} \\
  & P \in \psd_{++}^n, \ Z \in \psd_+^r \label{eq:lpv_qmi_alg_p_h2}\\
    & \alpha \in \R_{\geq 0}^{N_v}, \ \beta \in \R_{> 0}^{N_v}, \\
    & \textrm{LMIs \eqref{eq:h2_z} and \eqref{eq:h2_lmi_gamma} hold}  & & \forall v = 1..N_v. \label{eq:lpv_qmi_alg_gamma_h2}
\end{align}
\end{subequations}
\rev{The resultant $H_2$ norm is upper-bounded by $\gamma = \sqrt{\Tr{Z}}$ when using gain-scheduled control matrices $\{K_v = S_v P^{-1}\}_{v=1}^{N_v}$.}
All results in this section may be extended to \rev{discrete}-time $H_2$ control with appropriate \acp{LMI}.




\section{Numerical Examples}

\label{sec:examples}

Experiments were written in Matlab R2021a and are available at \url{https://github.com/jarmill/lpv_qmi} in the folder \texttt{experiments}. Dependencies include Mosek \cite{mosek92} and YALMIP \cite{Lofberg2004}. 
\rev{For both examples, the problem of finding a $\theta$-independent controller $K^c \in \R^{m \times n}$ with $\forall v: \ K_v = K^c$ that stabilizes all plants $(\{A_\ell\}, B)$ in the consistency set   $\Sigma_\dc$ is infeasible.}


\subsection{Two-Parameter, Two-State}
The \rev{experiment ground truth} with 
$\Theta = [0,2] \times [-1,1]$ is
\begin{align}
    A_1^{\textrm{true}} &= \begin{bmatrix} -0.2396  & -0.5845 \\ 0.5845 & -0.2396 \end{bmatrix} \nonumber &     A_2^{\textrm{true}} &= \begin{bmatrix} -0.1696  & 0.8434 \\ 0.8434 & 0.4140 \end{bmatrix} \nonumber\\
    B^{\textrm{true}} &= \begin{bmatrix} 0& -1.0072 \\ 0.4848 & 0 \end{bmatrix}\label{eq:two_state}
\end{align}

The plant $A_2$ in \eqref{eq:two_state} is open-loop unstable for both continuous-time and discrete-time with eigenvalues of $-0.7703, 1.0146$. \rev{Data $\dc$ with $T=35$ was collected}
under an individual-sample noise bound of $\epsilon = 0.1$. 
\subsubsection{Continuous-Time}
\rev{Eq. \eqref{eq:lpv_qmi_alg_c} synthesizes the following continuous-time vertex-controllers
\begin{align}
    K_{(0,1)} &= \begin{bmatrix}
          -4.5348 & -10.0625 \\
    9.9319  &  6.7597
    \end{bmatrix}   \nonumber\\  
    K_{(0,-1)} &= \begin{bmatrix}
             -4.7998 & -10.5553 \\
   10.7794  &  7.1231
    \end{bmatrix} \nonumber \\
        K_{(2,1)} &= \begin{bmatrix}
   -4.7566  & -9.8257\\
    9.5553 &   6.4091
    \end{bmatrix} \label{eq:two_state_k} \\        K_{(2,-1)} &= \begin{bmatrix}
       -4.7646  & -9.7462\\
    9.8597   & 6.4104
    \end{bmatrix}. \nonumber
\end{align}}
The \ac{LMI} parameters associated with $K$ in \eqref{eq:two_state_k} are
\rev{\begin{align}
    P &= \begin{bmatrix}
   0.0738& -0.0149\\ -0.0149& 0.0361
    \end{bmatrix} \in \psd_{++}^2 \label{eq:two_state_param}\\
    \alpha &= [  0.0535,
    0.0577,
    0.0518,
    0.0530] \in \R^4_{\geq 0} \nonumber \\
    \beta &= [10^{-5},10^{-5},10^{-5},10^{-5}] \in \R^4_{> 0}. \nonumber 
\end{align}}
The blue trajectories in Figure \ref{fig:param_seq} are system executions from 15 plants in the set $([A_1, A_2], B) \in \Sigma_\dc$ starting from the point $x(0) = [-2; 1.5]$. \rev{The parameter values $\theta$ are drawn uniformly from $[0,2]\times[-1,1]$ with exponentially distributed switching times (mean switching time is $0.05$).}
The red dotted-line in the top plot is the ground truth system from \eqref{eq:two_state} given the fixed parameter sequence. The bottom plot \rev{contains system trajectories for 30 parameter sequences on the ground truth and each of the 15 sampled plants}.

\begin{figure}[!ht]
    \centering
      \includegraphics[width=0.5\linewidth]{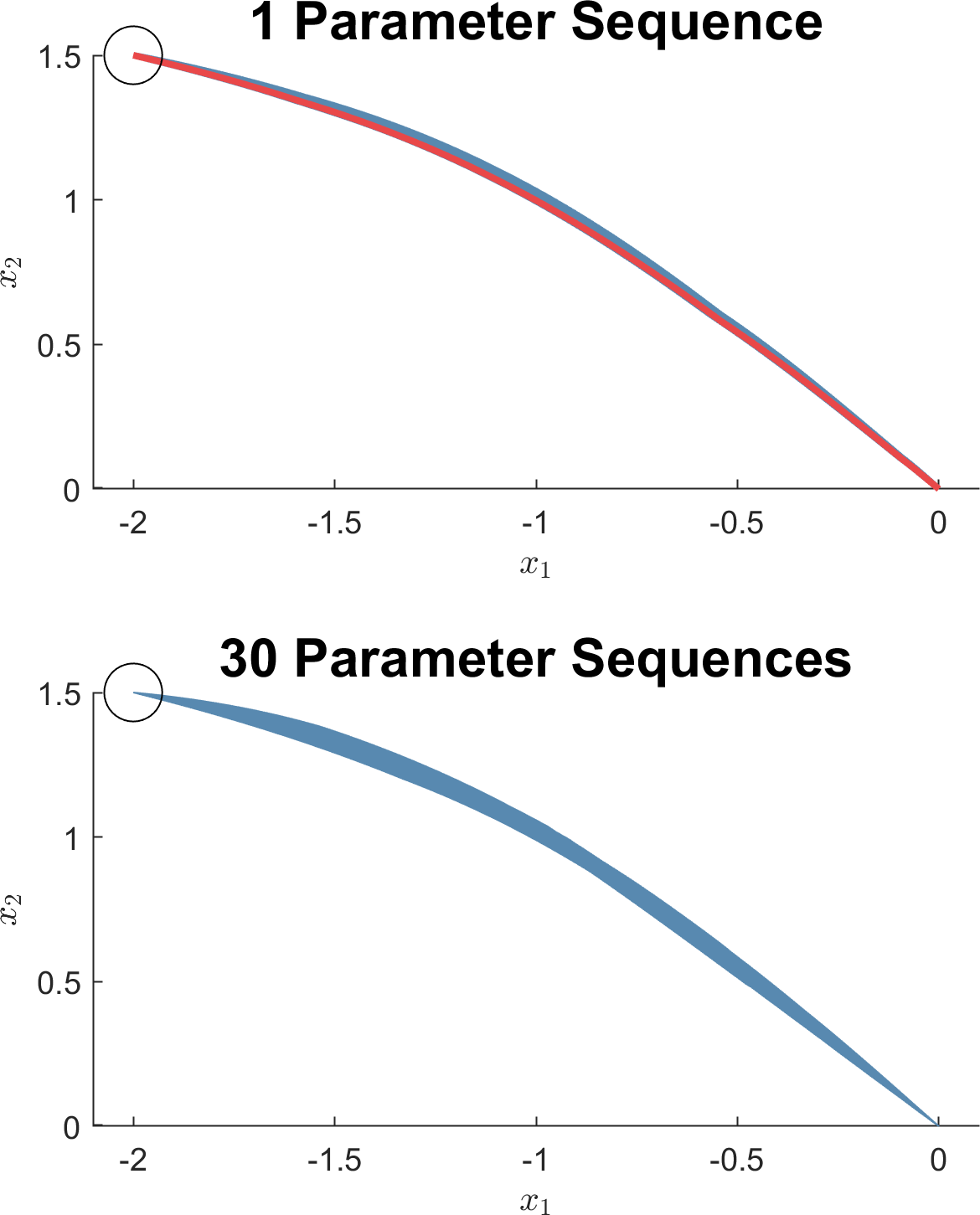}
    \caption{Plots of controlled trajectories using \eqref{eq:two_state_k}}
    \label{fig:param_seq}
\end{figure}

\subsubsection{Discrete-Time}
Eq. \eqref{eq:lpv_qmi_alg_d} with the same data $X_\delta$ creates the following discrete-time vertex-controllers
\begin{align}
    K_{\theta=(0,1)} &= \begin{bmatrix}
          -1.2258&   -0.6755 \\
    -0.1672   & 0.7948
    \end{bmatrix} \nonumber\\
    K_{\theta=(0,-1)} &= \begin{bmatrix}
          1.2258&   0.6755 \\
    0.1672   & -0.7948
    \end{bmatrix} \nonumber \\
        K_{\theta=(2,1)} &= \begin{bmatrix}
    -3.4132  & 0.1113 \\
    -0.6730  &    -0.3555
    \end{bmatrix}\label{eq:two_state_k_d} \\
        K_{\theta=(2,-1)} &= \begin{bmatrix}
    -0.5723&   1.4858 \\
    -0.3528   & -1.9440
    \end{bmatrix}. \nonumber
\end{align}

\rev{The resultant $P$ matrix is $[0.0588,   0.0014; 0.0014, 0.1022]$.

Figure \ref{fig:param_seq_discrete} visualizes a discrete-time trajectory of the ground truth ground-truth and 15 sample plants when the controller \eqref{eq:two_state_k_d} is applied to a single parameter sequence starting at $x(0) = [-2, 1.5]$. The bottom plot displays a sampled reachable set attained from 30 parameter sequences and all plants (15 sample plants plus ground truth).
\begin{figure}[!ht]
    \centering
    \includegraphics[width=0.5\linewidth]{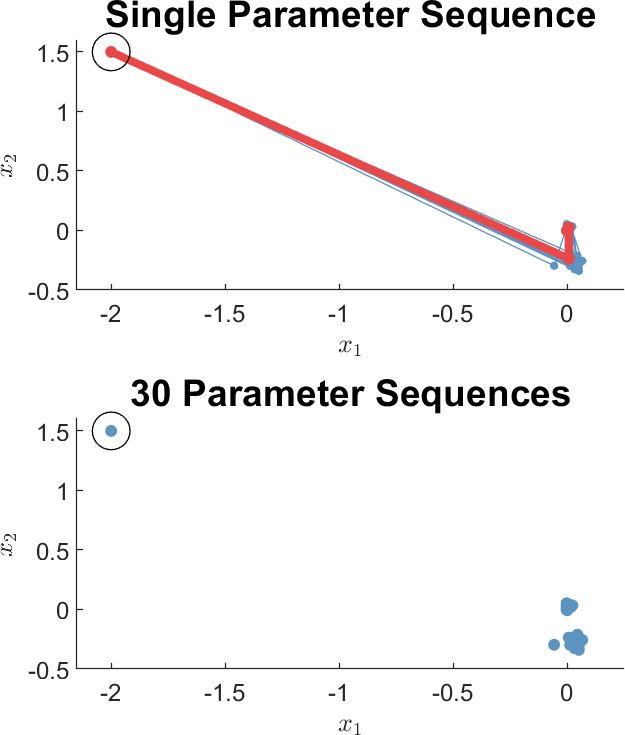}
    \caption{Plots of controlled trajectories using \eqref{eq:two_state_k}}
    \label{fig:param_seq_discrete}
\end{figure}

}
The discrete-time worst-case controlled $H_2$ norm with $C = [I_2; \0_2], \ D = [\0_2; \sqrt{2} I_2], \ F = I_2$ is bounded by $\gamma = 9.334$ by \rev{Eq. \eqref{eq:lpv_qmi_h2}}. 



\subsection{Three-Parameter, Five-State}
The second experiment involves a system with $n=5, m=3, L=3$. The parametric set is $\Theta = [-0.3, 0.3] \times [0.2, 0.8] \times [0.5, 1.5]$ with $N_v = 8$. A trajectory is recorded with a time horizon of $T = 50$ and an individual-sample noise bound of $\epsilon = 0.1$. The associated $P$ matrix to the controller is
 
\begin{align}
    P &= \begin{bmatrix}
    0.268  & 0.141  & -0.112 & 0.113  & -0.171 \\
0.141  & 0.322  & -0.151 & 0.250  & -0.277 \\
-0.112 & -0.151 & 0.410  & 0.210  & -0.130 \\
0.113  & 0.250  & 0.210  & 1.226  & -1.138 \\
-0.171 & -0.277 & -0.123 & -1.138 & 1.191
    \end{bmatrix}. \nonumber
\end{align}



\section{Conclusion}

\label{sec:conclusion}

This work considered quadratic stabilization of all \ac{LPV} systems $(\aell, B) \in \Sigma_\dc(\Phi)$. \rev{\acp{SDP} \eqref{eq:lpv_qmi_alg_c} and \eqref{eq:lpv_qmi_alg_d}} perform this task by solving a set of $N_v+1$ \acp{LMI} in order to recover a gain-scheduled controller. The unknown \ac{LPVA} plants may be regulated using a worst-case $H_2$-optimal controller. Sparsity of the \acp{LMI} may be employed to speed up computation of these controllers.
Future work involves finding $K(\theta)$ policies using methods that scale based on the number of faces of $\Theta$ rather than on $N_v$ and reducing the conservatism of $K(\theta)$-controllers by letting $P$ depend on $\theta$\rev{.}







\bibliographystyle{IEEEtran}
\bibliography{references.bib}

\begin{thebibliography}{10}
\providecommand{\url}[1]{#1}
\csname url@samestyle\endcsname
\providecommand{\newblock}{\relax}
\providecommand{\bibinfo}[2]{#2}
\providecommand{\BIBentrySTDinterwordspacing}{\spaceskip=0pt\relax}
\providecommand{\BIBentryALTinterwordstretchfactor}{4}
\providecommand{\BIBentryALTinterwordspacing}{\spaceskip=\fontdimen2\font plus
\BIBentryALTinterwordstretchfactor\fontdimen3\font minus
  \fontdimen4\font\relax}
\providecommand{\BIBforeignlanguage}[2]{{%
\expandafter\ifx\csname l@#1\endcsname\relax
\typeout{** WARNING: IEEEtran.bst: No hyphenation pattern has been}%
\typeout{** loaded for the language `#1'. Using the pattern for}%
\typeout{** the default language instead.}%
\else
\language=\csname l@#1\endcsname
\fi
#2}}
\providecommand{\BIBdecl}{\relax}
\BIBdecl

\bibitem{sename2013robust}
O.~Sename, P.~Gaspar, and J.~Bokor, \emph{Robust Control and Linear Parameter
  Varying Approaches Application to Vehicle Dynamics}.\hskip 1em plus 0.5em
  minus 0.4em\relax Springer, 2013, vol. 437.

\bibitem{apkarian1995self}
P.~Apkarian, P.~Gahinet, and G.~Becker, ``Self-scheduled {H}-infinity control
  of linear parameter-varying systems: a design example,'' \emph{Automatica},
  vol.~31, no.~9, pp. 1251--1261, 1995.

\bibitem{bachnas2014review}
A.~Bachnas, R.~T{\'o}th, J.~Ludlage, and A.~Mesbah, ``A review on data-driven
  linear parameter-varying modeling approaches: A high-purity distillation
  column case study,'' \emph{Journal of Process Control}, vol.~24, no.~4, pp.
  272--285, 2014.

\bibitem{rugh2000research}
W.~J. Rugh and J.~S. Shamma, ``Research on gain scheduling,''
  \emph{Automatica}, vol.~36, no.~10, pp. 1401--1425, 2000.

\bibitem{becker1993control}
G.~Becker, A.~Packard, D.~Philbrick, and G.~Balas, ``Control of
  parametrically-dependent linear systems: A single quadratic lyapunov
  approach,'' in \emph{1993 American Control Conference}.\hskip 1em plus 0.5em
  minus 0.4em\relax IEEE, 1993, pp. 2795--2799.

\bibitem{packard1994gain}
A.~Packard, ``Gain scheduling via linear fractional transformations,''
  \emph{Systems \& control letters}, vol.~22, no.~2, pp. 79--92, 1994.

\bibitem{scherer2001lpvqmi}
\BIBentryALTinterwordspacing
C.~Scherer, ``{LPV} control and full block multipliers,'' \emph{Automatica},
  vol.~37, no.~3, pp. 361--375, 2001. [Online]. Available:
  \url{https://www.sciencedirect.com/science/article/pii/S000510980000176X}
\BIBentrySTDinterwordspacing

\bibitem{HOU20133}
Z.-S. Hou and Z.~Wang, ``From model-based control to data-driven control:
  Survey, classification and perspective,'' \emph{Information Sciences}, vol.
  235, pp. 3--35, 2013, data-based Control, Decision, Scheduling and Fault
  Diagnostics.

\bibitem{hou2017datasurvey}
Z.~Hou, H.~Gao, and F.~L. Lewis, ``{Data-Driven Control and Learning
  Systems},'' \emph{IEEE Transactions on Industrial Electronics}, vol.~64,
  no.~5, pp. 4070--4075, 2017.

\bibitem{willems2005note}
J.~C. Willems, P.~Rapisarda, I.~Markovsky, and B.~L. De~Moor, ``A note on
  persistency of excitation,'' \emph{Systems \& Control Letters}, vol.~54,
  no.~4, pp. 325--329, 2005.

\bibitem{depersis2020formulas}
C.~De~Persis and P.~Tesi, ``{Formulas for Data-Driven Control: Stabilization,
  Optimality, and Robustness},'' \emph{IEEE Trans. Automat. Contr.}, vol.~65,
  no.~3, pp. 909--924, 2020.

\bibitem{coulson2019data}
J.~Coulson, J.~Lygeros, and F.~D{\"o}rfler, ``Data-enabled predictive control:
  {I}n the shallows of the {DeePC},'' in \emph{2019 18th European Control
  Conference (ECC)}.\hskip 1em plus 0.5em minus 0.4em\relax IEEE, 2019, pp.
  307--312.

\bibitem{berberich2020mpc}
J.~Berberich, J.~Köhler, M.~A. Müller, and F.~Allgöwer, ``{Data-Driven Model
  Predictive Control With Stability and Robustness Guarantees},'' \emph{IEEE
  Trans. Automat. Contr.}, vol.~66, no.~4, pp. 1702--1717, 2021.

\bibitem{waarde2020noisy}
H.~J. van Waarde, M.~K. Camlibel, and M.~Mesbahi, ``From noisy data to feedback
  controllers: non-conservative design via a matrix {S}-lemma,'' \emph{IEEE
  Trans. Automat. Contr.}, 2020.

\bibitem{yakubovich1997s}
V.~A. Yakubovich, ``{S-Procedure in Nonlinear Control Theory},'' \emph{Vestnick
  Leningrad Univ. Math.}, vol.~4, pp. 73--93, 1997.

\bibitem{waarde2022qmi}
\BIBentryALTinterwordspacing
H.~J. van Waarde, M.~K. Camlibel, J.~Eising, and H.~L. Trentelman, ``Quadratic
  matrix inequalities with applications to data-based control,'' 2022.
  [Online]. Available: \url{https://arxiv.org/abs/2203.12959}
\BIBentrySTDinterwordspacing

\bibitem{dai2021nonlinear}
T.~Dai and M.~Sznaier, ``{Nonlinear Data-Driven Control via State-Dependent
  Representations},'' in \emph{2021 60th IEEE Conference on Decision and
  Control (CDC)}.\hskip 1em plus 0.5em minus 0.4em\relax IEEE, 2021, pp.
  5765--5770.

\bibitem{berberich2020combining}
J.~Berberich, C.~W. Scherer, and F.~Allg{\"o}wer, ``{Combining Prior Knowledge
  and Data for Robust Controller Design},'' \emph{arXiv preprint
  arXiv:2009.05253}, 2020.

\bibitem{BERBERICH2021210}
J.~Berberich, S.~Wildhagen, M.~Hertneck, and F.~Allgöwer, ``Data-driven
  analysis and control of continuous-time systems under aperiodic sampling,''
  \emph{IFAC-PapersOnLine}, vol.~54, no.~7, pp. 210--215, 2021, 19th IFAC
  Symposium on System Identification SYSID 2021.

\bibitem{eising2022informativity}
J.~Eising and J.~Cort{\'e}s, ``Informativity for centralized design of
  distributed controllers for networked systems,'' in \emph{2022 European
  Control Conference (ECC)}.\hskip 1em plus 0.5em minus 0.4em\relax IEEE, 2022,
  pp. 681--686.

\bibitem{bisoffi2022learning}
A.~Bisoffi, C.~De~Persis, and P.~Tesi, ``Learning controllers for performance
  through lmi regions,'' \emph{IEEE Transactions on Automatic Control}, 2022.

\bibitem{formentin2016direct}
S.~Formentin, D.~Piga, R.~T{\'o}th, and S.~M. Savaresi, ``Direct learning of
  {LPV} controllers from data,'' \emph{Automatica}, vol.~65, pp. 98--110, 2016.

\bibitem{piga2018direct}
D.~Piga, S.~Formentin, and A.~Bemporad, ``Direct data-driven control of
  constrained systems,'' \emph{IEEE Transactions on Control Systems
  Technology}, vol.~26, no.~4, pp. 1422--1429, 2018.

\bibitem{verhoek2021fundamental}
C.~Verhoek, R.~T{\'o}th, S.~Haesaert, and A.~Koch, ``{Fundamental Lemma for
  Data-Driven Analysis of Linear Parameter-Varying Systems},'' in \emph{2021
  60th IEEE Conference on Decision and Control (CDC)}.\hskip 1em plus 0.5em
  minus 0.4em\relax IEEE, 2021, pp. 5040--5046.

\bibitem{dai2018moments}
T.~Dai and M.~Sznaier, ``{A Moments Based Approach to Designing MIMO Data
  Driven Controllers for Switched Systems},'' in \emph{2018 IEEE Conference on
  Decision and Control (CDC)}.\hskip 1em plus 0.5em minus 0.4em\relax IEEE,
  2018, pp. 5652--5657.

\bibitem{khatri1968solutions}
C.~Khatri and C.~R. Rao, ``{Solutions to Some Functional Equations and Their
  Applications to Characterization of Probability Distributions},''
  \emph{Sankhy{\=a}: The Indian Journal of Statistics, Series A}, pp. 167--180,
  1968.

\bibitem{besselmannlofberg2012}
T.~Besselmann and J.~L{\"o}fberg, ``Explicit {MPC for LPV} systems: stability
  and optimality,'' \emph{IEEE Trans. Automat. Contr.}, 2012.

\bibitem{boyd2004convex}
S.~Boyd, S.~P. Boyd, and L.~Vandenberghe, \emph{{Convex Optimization}}.\hskip
  1em plus 0.5em minus 0.4em\relax Cambridge university press, 2004.

\bibitem{boyd1994linear}
S.~Boyd, L.~El~Ghaoui, E.~Feron, and V.~Balakrishnan, \emph{{Linear Matrix
  Inequalities in System and Control Theory}}.\hskip 1em plus 0.5em minus
  0.4em\relax SIAM, 1994.

\bibitem{waarde2020informativity}
H.~J. Van~Waarde, J.~Eising, H.~L. Trentelman, and M.~K. Camlibel, ``Data
  informativity: a new perspective on data-driven analysis and control,''
  \emph{IEEE Trans. Automat. Contr.}, vol.~65, no.~11, pp. 4753--4768, 2020.

\bibitem{alizadeh1995interior}
F.~Alizadeh, ``{Interior Point Methods in Semidefinite Programming with
  Applications to Combinatorial Optimization},'' \emph{SIAM J OPTIMIZ}, vol.~5,
  no.~1, pp. 13--51, 1995.

\bibitem{de2003parametric}
C.~De~Souza, A.~Trofino, and J.~De~Oliveira, ``Parametric lyapunov function
  approach to h2 analysis and control of linear parameter-dependent systems,''
  \emph{IEE Proceedings-Control Theory and Applications}, vol. 150, no.~5, pp.
  501--508, 2003.

\bibitem{mosek92}
\BIBentryALTinterwordspacing
M.~ApS, \emph{The MOSEK optimization toolbox for MATLAB manual. Version 9.2.},
  2020. [Online]. Available:
  \url{https://docs.mosek.com/9.2/toolbox/index.html}
\BIBentrySTDinterwordspacing

\bibitem{Lofberg2004}
J.~L{\"{o}}fberg, ``{YALMIP : A Toolbox for Modeling and Optimization in
  MATLAB},'' in \emph{In Proceedings of the CACSD Conference}, Taipei, Taiwan,
  2004.

\end{thebibliography}

\end{document}